\documentclass[11pt,a4paper]{article}

\usepackage[T1]{fontenc}
\usepackage[utf8]{inputenc}
\usepackage{lmodern}
\usepackage[a4paper,margin=28mm]{geometry}
\usepackage{microtype}
\usepackage{mathtools,amssymb,amsthm}
\usepackage{booktabs,tabularx,array}
\usepackage{enumitem}
\usepackage{graphicx}
\usepackage[font=small,labelfont=bf]{caption}
\usepackage{subcaption}
\usepackage{tikz}
\usepackage{xcolor}
\usepackage{listings}
\usepackage[hidelinks]{hyperref}
\usepackage[nameinlink,noabbrev,capitalise]{cleveref}

\numberwithin{equation}{section}
\newtheorem{theorem}{Theorem}[section]
\newtheorem{lemma}[theorem]{Lemma}
\newtheorem{proposition}[theorem]{Proposition}
\newtheorem{corollary}[theorem]{Corollary}
\newtheorem{conjecture}[theorem]{Conjecture}

\theoremstyle{definition}

\theoremstyle{remark}

\newcommand{\V}{V}
\newcommand{\E}{E}
\newcommand{\N}{N}
\newcommand{\D}{\Delta}
\newcommand{\chiD}{\chi'_D}
\newcommand{\qB}{q_B}
\newcommand{\cB}{\mathcal{B}}
\newcommand{\cC}{\mathcal{C}}
\newcommand{\col}{\varphi}
\newcommand{\eps}{\varepsilon}
\DeclareMathOperator{\Int}{int}

\setlist{topsep=2pt,itemsep=1pt,parsep=0pt}
\title{D-coloring of planar graphs}
\author{%
Xiaoxue Hu\textsuperscript{1},
Jiangxu Kong\textsuperscript{2},
and Yiqiao Wang\textsuperscript{3}\thanks{\footnotesize Research supported  by  National Natural Science Foundation of China (Nos.\,12422113)},\\[0.5em]
\small\textsuperscript{1}School of Science,
Zhejiang University of Science and Technology,
Hangzhou 310023, China\\
\small\textsuperscript{2}School of Mathematics,
Hangzhou Normal University,
Hangzhou, China\\
\small\textsuperscript{3}Department of Mathematics,
Beijing University of Technology,
Beijing 100124, China
}
\date{}
\hypersetup{
  pdftitle={The D-Chromatic Index of Planar Graphs with Small or Large Maximum Degree},
  pdfsubject={D-colorings of planar graphs},
  pdfkeywords={D-coloring, rainbow diamond, planar graph, edge-coloring}
}

\begin{document}
\maketitle
\vspace{-2.5em}

\begin{abstract}
A proper edge-coloring of a graph $G$ is a D-coloring if every subgraph
isomorphic to $K_4-e$ is rainbow. The minimum number of colors in such a
coloring is the D-chromatic index $\chiD(G)$. Wang conjectured that every
planar graph of maximum degree $\D\ge4$ satisfies $\chiD(G)\le9$ for
$\D=4$, $\chiD(G)\le10$ for $\D=5$, and $\chiD(G)\le2\D-1$ for
$\D\ge6$. We prove that every planar graph $G$ satisfies
\[
 \chiD(G)\le
 \begin{cases}
  9, & \D(G)\le4,\\
  10, & \D(G)=5,\\
  2\D(G)-1, & \D(G)\ge33.
 \end{cases}
\]
Each bound is best possible in its stated range. Consequently, Wang's conjecture remains open only for
$6\le\D\le32$.
\end{abstract}



\section{Introduction}\label{sec:intro}

 All graphs considered in this paper are finite and simple. For a graph
$G$, let $\V(G)$ and $\E(G)$ denote its vertex set and edge set,
respectively, and let $\delta(G)$ and $\D(G)$ denote its minimum and
maximum degrees. We write $\D=\D(G)$ when the underlying graph is clear.
For a positive integer $k$, let $[k]=\{1,\ldots,k\}$. A proper
edge-$k$-coloring of $G$ is a mapping
\(
 \col:\E(G)\longrightarrow [k]
\)
such that adjacent edges receive distinct colors. The chromatic index of
$G$ is denoted by $\chi'(G)$. A subgraph of an edge-colored graph is
\emph{rainbow} if its edges receive pairwise distinct colors.

A \emph{diamond} is a graph isomorphic to $K_4-e$. A proper edge-coloring
of $G$ is a \emph{D-coloring} if every diamond subgraph of $G$ is
rainbow. The minimum number of colors in a D-coloring of $G$ is the
\emph{D-chromatic index}, denoted by $\chiD(G)$. Equivalently, a proper
edge-coloring is a D-coloring if, for every pair of nonincident edges
$ab,cd\in \E(G)$ satisfying
\(
 \bigl|\{ac,ad,bc,bd\}\cap \E(G)\bigr|\ge 3,
\)
the edges $ab$ and $cd$ receive distinct colors.

D-coloring was introduced by Wang~\cite{Wang2026}, motivated by the
B-coloring of Gy\'arf\'as and S\'ark\"ozy~\cite{Gyarfas2023}.
The latter was introduced as one of several ``less strong'' edge
colorings arising from graph formulations of the $(7,4)$-conjecture.
A B-coloring is a proper edge-coloring in which every $4$-cycle is
rainbow, and the corresponding minimum number of colors is denoted by
$\qB(G)$; see also~\cite{Gyarfas2024,Kong2026}.

A diamond is a chorded $4$-cycle, and in any proper edge-coloring its
chord already has a color distinct from those of the cycle edges.
Thus D-coloring imposes the B-coloring condition exactly on chorded
$4$-cycles, or equivalently, on pairs of triangles sharing an edge.

This condition also admits a useful conflict-graph formulation:
$\chiD(G)$ is the chromatic number of the graph with vertex set
$\E(G)$ in which two edges are adjacent whenever they are incident in
$G$ or are opposite edges of a common diamond. This graph contains
$L(G)$ and is a spanning subgraph of $L(G)^2$. Consequently,
\begin{equation}\label{eq:parameter-chain}
 \chi'(G)\le \chiD(G)\le \qB(G)\le \chi_s'(G),
\end{equation}
where $\chi_s'(G)$ denotes the strong chromatic index of $G$.

If no edge of $G$ lies in two triangles, then $G$ is diamond-free and
$\chiD(G)=\chi'(G)$; if $G$ is chordal, every $4$-cycle is chorded and
$\chiD(G)=\qB(G)$. Thus D-coloring does not retain the full connection
of B-coloring with the $(7,4)$-conjecture: every bipartite graph satisfies
$\chiD(G)=\chi'(G)$. Instead, it focuses on conflicts arising from
overlapping triangles.

This triangle-based interpretation naturally leads to the classical
notion of booksize. The \emph{book number} of a graph $G$ is
\(
 \operatorname{bk}(G)
   :=\max_{xy\in \E(G)}|N_G(x)\cap N_G(y)|,
\)
that is, the largest number of triangles sharing a common edge. The
study of this parameter goes back to the book problem of Erd\H{o}s. A
theorem of Edwards and, independently, Khad\v{z}iivanov and Nikiforov
states that every $n$-vertex graph with more than
$\lfloor n^2/4\rfloor$ edges has book number at least $n/6$; see
\cite{Khadz1979,Bollobas2005}.

The book number gives a direct lower bound for the D-chromatic index.
Suppose that an edge $xy$ has distinct common neighbors
$z_1,\ldots,z_t$. Then the $2t+1$ edges
\(
 \{xy\}\cup\{xz_i,yz_i:1\le i\le t\}
\)
must receive pairwise distinct colors in every D-coloring. Indeed, the
only nontrivial pairs are $xz_i$ and $yz_j$ with $i\ne j$, and these are
opposite edges of a diamond with spine $xy$. Similarly, all the edges of
a complete subgraph must receive distinct colors: two such edges are
either incident or are opposite edges of a diamond. We therefore have
the general lower bound
\begin{equation}\label{eq:basic-lower-bound}
 \chiD(G)\ge
 \max\left\{
   \chi'(G),\,
   2\operatorname{bk}(G)+1,\,
   \binom{\omega(G)}2
 \right\},
\end{equation}
where $\omega(G)$ is the clique number of $G$.

For unrestricted graphs, the clique term in
\eqref{eq:basic-lower-bound} can be quadratic in the maximum degree.
In particular,
\(
 \chiD(K_{\D+1})=\binom{\D+1}{2}.
\)
Motivated by this example, Wang~\cite{Wang2026} conjectured that every
graph of maximum degree $\D$ satisfies
\(
 \chiD(G)\le \binom{\D+1}{2}.
\)
He verified this conjecture for $\D\le 5$ and proved the general upper
bound
\[
 \chiD(G)\le
 \left\lfloor
   \frac{9}{16}\D^2+\frac12\D
 \right\rfloor.
\]

The extremal picture is different for planar graphs. Since
$\omega(G)\le 4$, the clique term in \eqref{eq:basic-lower-bound} is at
most $6$, whereas the book term can grow linearly with $\D$. For
$\D\ge 2$, let $F_\D$ be the graph obtained from $K_{2,\D-1}$ by adding
the edge joining the two vertices in the part of size $2$. The graph
$F_\D$ is planar, has maximum degree $\D$, and all its $2\D-1$ edges
must receive distinct colors. Hence
\begin{equation}\label{eq:planar-book-lower-bound}
 \chiD(F_\D)=2\D-1.
\end{equation}
Thus the value $2\D-1$ is forced by a single maximum book.

The corresponding B-coloring problem has already been studied for
planar graphs. Gy\'arf\'as, Martin, Ruszink\'o, and
S\'ark\"ozy~\cite{Gyarfas2024} proved that
\(
 \qB(G)\le 2\D+8
\)
for every planar graph $G$. This was subsequently improved by Kong, Wang and Zheng in
\cite{Kong2026}, where it was shown, in particular, that
\(\qB(G)\le 2\D\) when $\D\ge 38$.
By \eqref{eq:parameter-chain}, the latter result also gives
$\chiD(G)\le 2\D$ in this range. The book construction
\eqref{eq:planar-book-lower-bound} suggests that the extra color is
unnecessary for D-coloring.

There are two additional obstructions at small maximum degree. The
planar graph $K_2\vee P_3$ has maximum degree $4$, and its nine edges
are pairwise conflicting in every D-coloring. For maximum degree $5$,
let $x,y$ be the vertices of $K_2$ and let
$z_1z_2z_3z_4$ be the path in $K_2\vee P_4$. The ten edges
\(
 \{xy,z_2z_3\}
 \cup
 \{xz_i,yz_i:1\le i\le4\}
\)
must receive pairwise distinct colors. These examples, together with
$F_\D$, led Wang to the following conjecture.

\begin{conjecture}[Wang~\cite{Wang2026}]\label{conj1.1}
Let $G$ be a planar graph with maximum degree $\D\ge4$. Then
\[
 \chiD(G)\le
 \begin{cases}
  9,       & \D=4,\\
  10,      & \D=5,\\
  2\D-1,   & \D\ge6.
 \end{cases}
\]
\end{conjecture}

Our main theorem settles both exceptional small-degree cases and proves
the conjectured bound for all sufficiently large maximum degrees.

\begin{theorem}\label{thm:main}
Let $G$ be a planar graph of maximum degree $\D$. Then
\[
 \chiD(G)\le
 \begin{cases}
  9,       & \D\le4,\\
  10,      & \D=5,\\
  2\D-1,   & \D\ge33.
 \end{cases}
\]
 \end{theorem}

Consequently, \cref{conj1.1} remains open only for
\(
 6\le \D\le32.
\)
In particular, \cref{thm:main} shows that, for sufficiently large
maximum degree, the largest D-chromatic index among planar graphs is
already attained by a single maximum book. 

The three ranges require different methods. For $\D\le4$, we use a minimal counterexample argument and analyze the induced subgraphs of vertex neighborhoods, eventually reducing to the octahedral graph. For
$\D=5$, we develop a local patch recoloring method combining list degeneracy with graph polynomial certificates. For $\D\ge33$, we combine the stars and bunches lemma with blocker support estimates and a common palette recoloring argument for bunches.

The paper is organized as follows. In \cref{sec:prelim}, we introduce the local conflict notation,
establish the extension tools, and state the stars and bunches lemma. The cases $\D\le4$ and $\D=5$ are proved in
\cref{sec:four,sec:five}, respectively. The proof for $\D\ge33$ is given in \cref{sec:large}. 
The code used to verify the graph-polynomial coefficients appears in
Appendix~\ref{app:code}.

\section{Preliminaries}\label{sec:prelim}

Two distinct edges $e,f\in\E(G)$ are said to \emph{see} each other if
they are incident, or if they are nonincident and lie in a common
diamond. Thus a proper edge-coloring is a D-coloring precisely when
every two edges that see each other receive distinct colors.

Since $G-S$ is an induced subgraph, two edges of $G-S$ see each other
in $G-S$ if and only if they see each other in $G$. Indeed, incidence
is unchanged, and for two nonincident edges $ab$ and $cd$, the relevant
cross-edges $ac,ad,bc,bd$ all have both endpoints outside $S$.
Consequently, every D-coloring of $G-S$ may be viewed as a partial
D-coloring of $G$.

For a partial D-coloring $\col$ of $G$ and a vertex $x$, let
$C_\col(x)$ denote the set of colors on the colored edges incident
with $x$. For an uncolored edge $e$, let $L_\col(e)$ denote the set of
colors not used on any colored edge that sees $e$ in $G$. When $\col$
is fixed, we simply write $C(x)$ and $L(e)$.

In each deletion and extension argument, we specify a local set of
edges to be colored or recolored. These are called the \emph{new edges}
(or \emph{patch edges}), even after some of them have been colored during
the extension; the colored edges outside this local set that retain the
inherited coloring are called the \emph{old edges}. For an uncolored edge
$e$, a \emph{blocker} of $e$ is a currently colored edge $f$ that sees
$e$, so the color of $f$ is excluded from $L_\col(e)$. A blocker is
\emph{incident} if it shares an endpoint with $e$, and \emph{nonincident}
otherwise. Thus a nonincident blocker is disjoint from $e$ and lies with
$e$ in a common diamond. An \emph{old blocker} is a blocker that is an
old edge.

We shall repeatedly use the following greedy extension lemma.

\begin{lemma}\label{lem2.1}
Let $e_1,\ldots,e_t$ be pairwise seeing uncolored edges. If they can be
ordered so that
\(
 |L(e_i)|\ge i
 \)
\(
(1\le i\le t),
\)
then they can be colored from their lists with pairwise distinct colors.
\end{lemma}

\begin{proof}
Color $e_1,\ldots,e_t$ successively. When $e_i$ is considered, at most
$i-1$ colors in $L(e_i)$ have been used on the preceding edges, so an
available color remains.
\end{proof}

We shall also use the following standard graph-polynomial consequence of the Combinatorial Nullstellensatz.

\begin{lemma}[Alon~\cite{Alon1999}]\label{lem:alon}
Let $Q$ be a graph whose vertices are ordered as $z_1<\cdots<z_n$, and define
\[
 P_Q(z_1,\ldots,z_n)=\prod_{z_iz_j\in E(Q),\ i<j}(z_i-z_j).
\]
If the coefficient of $\prod_{i=1}^n z_i^{t_i}$ in $P_Q$ is nonzero, then $Q$ is colorable from every list assignment satisfying $|L(z_i)|\ge t_i+1$ for all $i$.
\end{lemma}

\begin{lemma}\label{lem2.3}
Let $v\in\V(G)$, put $U=\N_G(v)$ and $H=G-v$, and fix $u\in U$. If an
old edge $f\in\E(H)$ sees $vu$, then $f$ has an endpoint in
\(
 W_u:=\{u\}\cup\N_{G[U]}(u).
\)
Consequently, the number of old edges seeing $vu$ is at most
\(
 \sum_{z\in W_u}d_H(z)-e_H(W_u).
\)
\end{lemma}

\begin{proof}
The assertion is clear if $f$ is incident with $u$. Otherwise, write
$f=ab$. Since $f$ sees $vu$, at least three of
\(
 va,\ vb,\ ua,\ ub
\)
are edges of $G$. Hence one of $a,b$ is adjacent to both $v$ and $u$, and so belongs to $\N_{G[U]}(u)$. Thus $f$ has an endpoint in $W_u$. The number of edges of $H$ having an endpoint in $W_u$ is
\[
 e_H(W_u)+e_H(W_u,\V(H)\setminus W_u)
 =\sum_{z\in W_u}d_H(z)-e_H(W_u),
\]
as required.
\end{proof}

We shall also use the following immediate observation. For
$xy\in\E(G)$, let
\[
 Z=\N_G(x)\cap\N_G(y)
 \quad\text{and}\quad
 \cB(xy)=\{xy\}\cup\{xz,yz:z\in Z\}.
\]
The edges of $\cB(xy)$ are pairwise seeing and hence receive distinct
colors in every D-coloring. Indeed, for distinct $z,w\in Z$, the edges
$xz$ and $yw$ are opposite edges of a diamond containing the three
cross-edges $xy,xw,yz$.

\begin{lemma}\label{lem2.4}
If $G$ is planar and $x\in\V(G)$, then $G[\N_G(x)]$ is outerplanar.
Consequently, it contains neither a $K_4$-minor nor a $K_{2,3}$-minor,
and any two of its vertices have at most two common neighbors.
\end{lemma}

\begin{proof}
Since
\(
 K_1\vee G[N_G(x)]\cong G[N_G[x]]
\)
is planar, the cone characterization of outerplanar graphs implies that
$G[N_G(x)]$ is outerplanar; see~\cite{Felsner2004}. The remaining
assertions follow from the forbidden-minor characterization of
outerplanar graphs: two vertices with three common neighbors would
yield a $K_{2,3}$ subgraph.
\end{proof}

As an immediate consequence of \cref{lem2.4}, the three vertices of
a triangle in a planar graph have at most two common neighbors.

\begin{lemma}\label{lem2.5}
If $G$ is planar and $xy\in\E(G)$, then
\(
 G[Z]
\)
is a linear forest.
\end{lemma}

\begin{proof}
For each $z\in Z$, the vertices $x,y,z$
form a triangle, and every neighbor of $z$ in $G[Z]$ is a common
neighbor of this triangle. Hence
\(
 \D(G[Z])\le2
\)
by \cref{lem2.4}.

Suppose that $G[Z]$ contains a cycle $C$. Partition $V(C)$ into three
nonempty consecutive intervals along $C$ and contract each interval to
a single vertex. The resulting three vertices form a triangle and are
all adjacent to both $x$ and $y$. Together with the edge $xy$, they
form a $K_5$-minor, contradicting planarity. Thus $G[Z]$ is acyclic
and has maximum degree at most $2$, and hence is a linear forest.
\end{proof}

Let $G$ be a plane graph. A \emph{bunch} $B(x,y;m)$ consists of $m$
paths $Q_1,\ldots,Q_m$, each of length $1$ or $2$, joining two poles
$x,y$, such that the cycle formed by consecutive paths $Q_i,Q_{i+1}$
is nonseparating and the sequence is maximal. If $Q_i=xz_i y$, then
$z_i$ is a \emph{brother}; if $Q_i=xy$, then $xy$ is a \emph{parental
edge}. A brother $z_i$ is internal when $2\le i\le m-1$, and strictly
internal when $3\le i\le m-2$. When a parental edge is present, its
two sides refer to the two directions in the cyclic order of the bunch
paths. See \cref{fig:bunch}.

We shall use the following standard geometric property of bunches
\cite{Borodin2001}.

\begin{lemma}\label{lem:bunch-geometry}
Every internal brother $z_i$ has
$N_G(z_i)\subseteq\{x,y,z_{i-1},z_{i+1}\}$ and hence degree at most
four. If $z_i$ is strictly internal, each brother adjacent to $z_i$ is
internal.
\end{lemma}

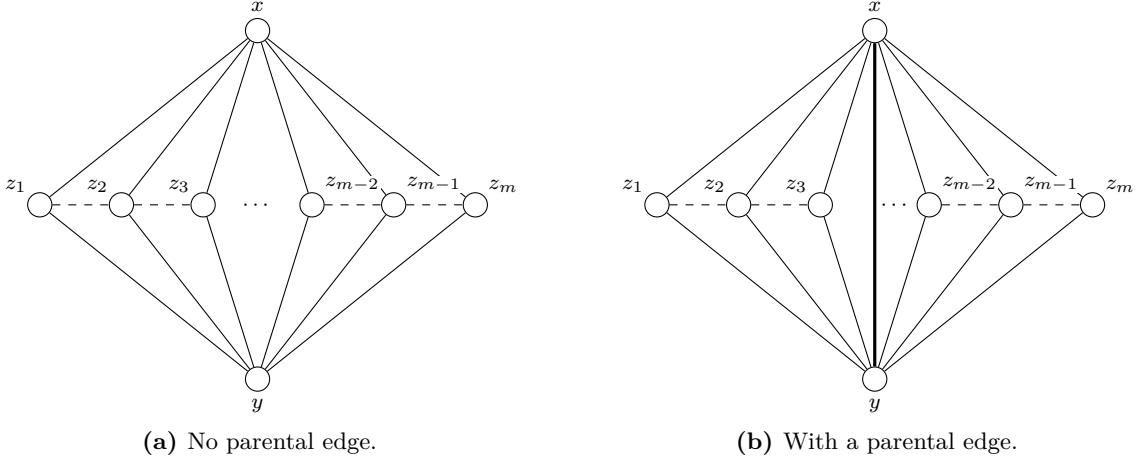
\begin{figure}[htbp]
\centering

\begin{subfigure}{0.47\textwidth}
\centering
\begin{tikzpicture}[
  scale=.72,
  v/.style={
    circle,
    draw,
    fill=white,
    minimum size=3.2mm,
    inner sep=0pt
  },
  lab/.style={
    font=\scriptsize,
    fill=white,
    inner sep=.4pt
  }
]
  \node[v] (x) at (0,3.2) {};
  \node[v] (y) at (0,-3.2) {};

  \node[v] (z1) at (-4.0,0) {};
  \node[v] (z2) at (-2.5,0) {};
  \node[v] (z3) at (-1.0,0) {};
  \node[v] (z4) at ( 1.0,0) {};
  \node[v] (z5) at ( 2.5,0) {};
  \node[v] (z6) at ( 4.0,0) {};

  \foreach \i in {1,2,3,4,5,6}{
    \draw (x)--(z\i)--(y);
  }

  \draw[dashed] (z1)--(z2)--(z3);
  \draw[dashed] (z4)--(z5)--(z6);

  \node[lab,anchor=south,yshift=2pt] at (x.north) {$x$};
  \node[lab,anchor=north,yshift=-2pt] at (y.south) {$y$};

  \node[
    lab,
    anchor=south east,
    xshift=-1pt,
    yshift=1pt
  ] at (z1.north west) {$z_1$};

  \node[
    lab,
    anchor=south east,
    xshift=-1pt,
    yshift=1pt
  ] at (z2.north west) {$z_2$};

  \node[
    lab,
    anchor=south east,
    xshift=-1pt,
    yshift=1pt
  ] at (z3.north west) {$z_3$};

  \node[
    lab,
    anchor=south west,
    xshift=1pt,
    yshift=1pt
  ] at (z4.north east) {$z_{m-2}$};

  \node[
    lab,
    anchor=south west,
    xshift=1pt,
    yshift=1pt
  ] at (z5.north east) {$z_{m-1}$};

  \node[
    lab,
    anchor=south west,
    xshift=1pt,
    yshift=1pt
  ] at (z6.north east) {$z_m$};

  \node[lab] at (0,0) {$\cdots$};
\end{tikzpicture}
\caption{No parental edge.}
\end{subfigure}
\hfill
\begin{subfigure}{0.47\textwidth}
\centering
\begin{tikzpicture}[
  scale=.72,
  v/.style={
    circle,
    draw,
    fill=white,
    minimum size=3.2mm,
    inner sep=0pt
  },
  lab/.style={
    font=\scriptsize,
    fill=white,
    inner sep=.4pt
  }
]
  \node[v] (x) at (0,3.2) {};
  \node[v] (y) at (0,-3.2) {};

  \node[v] (z1) at (-4.0,0) {};
  \node[v] (z2) at (-2.5,0) {};
  \node[v] (z3) at (-1.0,0) {};
  \node[v] (z4) at ( 1.0,0) {};
  \node[v] (z5) at ( 2.5,0) {};
  \node[v] (z6) at ( 4.0,0) {};

  \foreach \i in {1,2,3,4,5,6}{
    \draw (x)--(z\i)--(y);
  }

  \draw[very thick] (x)--(y);
  \draw[dashed] (z1)--(z2)--(z3);
  \draw[dashed] (z4)--(z5)--(z6);

  \node[lab,anchor=south,yshift=2pt] at (x.north) {$x$};
  \node[lab,anchor=north,yshift=-2pt] at (y.south) {$y$};

  \node[
    lab,
    anchor=south east,
    xshift=-1pt,
    yshift=1pt
  ] at (z1.north west) {$z_1$};

  \node[
    lab,
    anchor=south east,
    xshift=-1pt,
    yshift=1pt
  ] at (z2.north west) {$z_2$};

  \node[
    lab,
    anchor=south east,
    xshift=-1pt,
    yshift=1pt
  ] at (z3.north west) {$z_3$};

  \node[
    lab,
    anchor=south west,
    xshift=1pt,
    yshift=1pt
  ] at (z4.north east) {$z_{m-2}$};

  \node[
    lab,
    anchor=south west,
    xshift=1pt,
    yshift=1pt
  ] at (z5.north east) {$z_{m-1}$};

  \node[
    lab,
    anchor=south west,
    xshift=1pt,
    yshift=1pt
  ] at (z6.north east) {$z_m$};

  \node[lab,xshift=8pt] at (0,0) {$\cdots$};
\end{tikzpicture}
\caption{With a parental edge.}
\end{subfigure}

\caption{A bunch $B(x,y;m)$ in a plane graph.}
\label{fig:bunch}
\end{figure}


We use the following consequence of the stars-and-bunches theorem
of Borodin et al.~\cite{Borodin2001}. Taking a vertex to be big
when its degree is at least $26$, their theorem yields either a
precomplete star of weight at most $38$ whose prescribed neighbors
have degree at most $25$, or a bunch with a pole $x$ and at least
$d_G(x)/5$ paths. This gives the following formulation.

\begin{lemma}[Borodin et al.~\cite{Borodin2001}; see also~\cite{Kong2026}]\label{lem:structure}
Let $G$ be a plane graph with $\delta(G)\ge2$. Then $G$ contains one of the following configurations.
\begin{enumerate}[label=\textup{(B\arabic*)}]
\item A $k$-vertex $v$, where $2\le k\le5$, with neighbors $v_1,\ldots,v_k$ such that
\(
 d_G(v_i)\le25 (1\le i\le k-1),
 \)
\(
 \sum_{i=1}^{k-1}d_G(v_i)\le38.
\)
\item A bunch $B(x,y;m)$ satisfying
\(
 d_G(x)\ge26,
 \)
\(
 m\ge \frac{d_G(x)}5.
\)
\end{enumerate}
\end{lemma}

\section{The proof of $\Delta\le 4$}\label{sec:four}

Throughout this section, suppose that the first case of \cref{thm:main} is false, and let $G$ be a counterexample with the minimum number of vertices. We may assume that $G$ is connected.

\begin{lemma}\label{lem3.1}
The graph $G$ is $4$-regular.
\end{lemma}

\begin{proof}
Let $u$ be a vertex of degree $d(u)\le3$, and let $\col$ be a $9$-D-coloring of $G-u$. Fix $x\in\N(u)$ and put
\(
 r=|\N(u)\cap\N(x)|.
\)
At most $d_G(x)-1\le3$ old edges incident with $x$ see $ux$.

If $r=0$, then there is no nonincident blocker, and hence $|L(ux)|\ge6$.

If $r=1$, let $z$ be the common neighbor. Since $z$ is already adjacent to $u,x$ and has degree at most $4$, at most two edges at $z$ can be nonincident blockers. Thus $|L(ux)|\ge9-(3+2)=4$.

If $r=2$, then $d(u)=3$. Let $\N(u)=\{x,z_1,z_2\}$, and let $y$ be
the possible neighbor of $x$ outside $\{u,z_1,z_2\}$, if such a
neighbor exists. Every nonincident blocker is one of
\(
 z_1z_2, z_1y, z_2y,
\)
where the terms involving $y$ are omitted if $y$ does not exist. Thus
there are at most three such blockers, and
$|L(ux)|\ge9-(3+3)=3$.

Accordingly, each uncolored edge incident with $u$ has at least three
available colors. Hence \cref{lem2.1} extends the coloring to $G$,
contradicting the choice of $G$. Therefore $\delta(G)\ge4$; as
$\D(G)\le4$, the graph $G$ is $4$-regular.
\end{proof}

\begin{lemma}\label{lem3.2}
Let $G$ be a $4$-regular planar graph, let $u\in\V(G)$, and let
$\col$ be a D-coloring of $G-u$ with colors from $[9]$. Put
\(
 H=G[\N_G(u)].
\)
For $x\in\N_G(u)$, let
\(
 r=d_H(x)=|\N_G(u)\cap\N_G(x)|,
\)
and let $L_\col(ux)$ be the set of colors not used on any
$\col$-colored edge that sees $ux$ in $G$. Then
\[
\begin{array}{c|c|c}
 r & \text{condition on }N_H(x) & |L_\col(ux)| \\ \hline
 0 & \text{none} & \ge 6\\
 1 & \text{none} & \ge 4\\
 2 & \text{the two neighbors of $x$ in $H$ are adjacent} & \ge 3\\
 2 & \text{the two neighbors of $x$ in $H$ are nonadjacent} & \ge 2\\
 3 & \text{none} & \ge 4
\end{array}
\]
\end{lemma}

\begin{proof}
The cases $r=0,1$ were established in the proof of \cref{lem3.1}. Suppose $r=2$, with common neighbors $z_1,z_2$. Let $w$ be the fourth neighbor of $u$, and let $y$ be the unique neighbor of $x$ outside $\N[u]$. Obviously, the nonincident blockers are among
\[
 z_1z_2,\ z_1w,\ z_2w,\ z_1y,\ z_2y. \tag{3.1}
\]
If $z_1z_2\in\E(G)$, then each $z_i$ $(i=1,2)$ is already adjacent to $u,x,z_{3-i}$ and can be incident with at most one further edge in (3.1). Hence there are at most three nonincident blockers and $|L(ux)|\ge3$. If $z_1z_2\notin\E(G)$, then (3.1) contains at most four blockers and $|L(ux)|\ge2$.

Now let $r=3$. All nonincident blockers lie in $G[\N(u)\setminus\{x\}]$. If that graph were a triangle, then $G[\N[u]]$ would be $K_5$, which is impossible. Thus there are at most two nonincident blockers. Together with the three old edges at $x$, this gives $|L(ux)|\ge4$.
\end{proof}

Up to isomorphism, there are eleven simple graphs on four vertices in $H$. The type $K_4$ cannot occur as $H$, since it would give a $K_5$ on $\N[u]$. The remaining ten possibilities and the list lower bounds from \cref{lem3.2} are displayed in \cref{tab3.1}.

\begin{table}[htbp]
\centering
\small
\renewcommand{\arraystretch}{1.12}
\begin{tabular}{c c c}
\toprule
$H$ & degree sequence & sorted list lower bounds\\
\midrule
$4K_1$ & $(0,0,0,0)$ & $(6,6,6,6)$\\
$K_2\cup2K_1$ & $(0,0,1,1)$ & $(4,4,6,6)$\\
$P_3\cup K_1$ & $(0,1,1,2)$ & $(2,4,4,6)$\\
$2K_2$ & $(1,1,1,1)$ & $(4,4,4,4)$\\
$K_3\cup K_1$ & $(0,2,2,2)$ & $(3,3,3,6)$\\
$K_{1,3}$ & $(1,1,1,3)$ & $(4,4,4,4)$\\
$P_4$ & $(1,1,2,2)$ & $(2,2,4,4)$\\
paw & $(1,2,2,3)$ & $(3,3,4,4)$\\
$C_4$ & $(2,2,2,2)$ & $(2,2,2,2)$\\
$K_4-e$ & $(2,2,3,3)$ & $(3,3,4,4)$\\
\bottomrule
\end{tabular}
\caption{All possible four-vertex neighborhood graphs except $K_4$.}
\label{tab3.1}
\end{table}

\begin{proof}[Proof of the first case of \cref{thm:main}]
A minimum counterexample is $4$-regular by \cref{lem3.1}.  For every row of \cref{tab3.1} except $C_4$, the sorted lower bounds $b_1\le$ $b_2\le$ $b_3\le$ $b_4$ satisfy $b_i\ge i$. Apply \cref{lem2.1} to the four edges incident with $u$.  Since $u$ was arbitrary, minimality forces
$G[\N(u)]\cong C_4$ for every $u\in\V(G)$.  Choose $u$ and write $\N(u)=\{v_1$, $v_2$, $v_3$, $v_4\}$ so that $v_1v_2v_3v_4v_1$ is an induced $4$-cycle. 
Suppose that the neighborhood of $v_1$ is $\{u,v_2,v_4,q\}$. Here $uv_2,uv_4$ are present, while $v_2v_4,uq$ are absent. To induce a $C_4$, the remaining two edges must be $qv_2,qv_4$. Now
\(
 \N(v_2)=\{u,v_1,v_3,q\}.
\)
Within this neighborhood,  the $C_4$ condition forces $qv_3\in\E(G)$. It follows symmetrically that $q$ is adjacent to all four $v_i$. Thus, $G\cong K_{2,2,2}$. Assign colors as follows:  $uv_1, \ qv_3$ with color $1$, $uv_3,\ qv_1$ with color 2, $uv_2,\ qv_4$ with color 3, $uv_4,\ qv_2$
with color 4, $v_1v_2,\ v_3v_4$ with color 5, $v_1v_4,\ v_3v_2$ with color 6. We can check that the coloring is a D-coloring. 
 This contradiction proves the theorem.
\end{proof}

\section{The proof of $\Delta=5$}\label{sec:five}

In this section, we prove the second case of \cref{thm:main}; in fact, the argument proves that every planar graph of maximum degree at most five is $10$-D-colorable.

\begin{lemma}\label{lem4.1}
Let \(G\) be a planar graph with \(\D(G)\le5\), let
\(v\in V(G)\) satisfy \(d(v)\le3\), and let \(\varphi\) be a
\(10\)-D-coloring of \(H=G-v\). Put \(U=N_G(v)\). For
\(u\in U\), set $r_u=d_{G[U]}(u)$. 
Then \(r_u\in\{0,1,2\}\), and
\[
 |L(vu)|\ge
 \begin{cases}
  6,&r_u=0,\\
  3,&r_u=1,\\
  2,&r_u=2.
 \end{cases}
\]
\end{lemma}

\begin{proof}
Fix \(u\in U\), and put
 $X=U\setminus\{u\}$, $Y=N_G(u)\setminus\{v\}$ and
 $Z=X\cap Y.$
Thus \(|Z|=r_u\). There are at most \(d_G(u)-1\le4\) old
blockers of \(vu\) that are incident with \(u\).

Now let \(f\) be a nonincident old blocker of \(vu\). 
Since \(f\) and \(vu\) lie in a common diamond, its other endpoint
lies in \(X\cup Y\). Hence \(f\) either has both endpoints in \(Z\),
or has one endpoint in \(Z\) and the other in
\((X\cup Y)\setminus Z\).

If \(r_u=0\), there is no nonincident old blocker, and therefore
\(|L(vu)|\ge10-4=6\).

If \(r_u=1\), write \(Z=\{z\}\). The vertex \(z\) is already
adjacent to both \(v\) and \(u\), so it is incident with at most
three further edges. Thus \(vu\) has at most three nonincident old
blockers, and \(|L(vu)|\ge10-(4+3)=3\).

Suppose that \(r_u=2\). Since \(d_G(v)\le3\), this forces
\(d_G(v)=3\) and
$X=Z=\{z_1,z_2\}$. 
 Since \(d_G(u)\le5\), we have
\(|Y\setminus Z|\le2\), and every nonincident old blocker of \(vu\) belongs to
$E(G[Z])\cup E_G(Z,Y\setminus Z)$. 
There are at most four such blockers. Indeed, if there were five,
then the sets
$ \{u,z_1,z_2\}$ and $\{v,a_1,a_2\}$
would be the two partite sets of a \(K_{3,3}\), contrary to
planarity. Hence \(|L(vu)|\ge10-(4+4)=2\).
\end{proof}

\begin{lemma}\label{lem4.2}
Let \(G\) be a planar graph with \(\D(G)\le5\), and let
\(v\in V(G)\) satisfy \(d(v)\le3\). Every \(10\)-D-coloring of
\(G-v\) extends to a \(10\)-D-coloring of \(G\).
\end{lemma}

\begin{proof}
Let \(\varphi\) be a \(10\)-D-coloring of \(G-v\), put
\(U=N_G(v)\).

If \(d_G(v)\le2\), then \(d_{G[U]}(u)\le1\) for every \(u\in U\).
By \cref{lem4.1}, each new-edge list has size at least
three, 
\cref{lem2.1} extends \(\varphi\) to them.

Assume henceforth that \(d_G(v)=3\), and write
 $U=\{u_1,u_2,u_3\}$,  $L_i=L(vu_i)$, $r_i=d_{G[U]}(u_i)$. 
 
Suppose first that \(G[U]\ne K_3\). Two vertices of \(U\), say
\(u_1\) and \(u_2\), are nonadjacent, so \(r_1,r_2\le1\). By
\cref{lem4.1}, $|L_1|,|L_2|\ge3$ and $ |L_3|\ge2$.
Thus \cref{lem2.1} applies.

It remains to consider \(G[U]\cong K_3\). Here \(r_1=r_2=r_3=2\),
so \cref{lem4.1} gives \(|L_i|\ge2\) for every
\(i\in[3]\). By \cref{lem2.1}, we assume that
$ |L_1|=|L_2|=|L_3|=2$. 
Fix \(i\in[3]\), and let \(\{j,k\}=[3]\setminus\{i\}\). Since \(d_G(u_i)=5\),  \(vu_i\) has exactly four nonincident old
blockers.

The vertex \(u_i\) is already adjacent to \(v,u_j,u_k\), so it has
exactly two neighbors outside \(N_G[v]\); call them \(p_i\) and
\(q_i\). The possible nonincident blockers of \(vu_i\) are the edge
\(u_ju_k\) and the three possible edges between
\(\{u_j,u_k\}\) and \(\{p_i,q_i\}\). After interchanging \(p_i\) and
\(q_i\), if necessary, we may assume that
$p_i\text{ is adjacent to }u_j,u_k$,
 and  $q_i\text{ is adjacent to exactly one of }u_j,u_k$. Thus \(p_i\) is adjacent to all three vertices of \(U\).

If two of $p_1, p_2, p_3$ were distinct, those two vertices together with $v$ would form one part of a $K_{3,3}$, that's impossible. Hence $p_1=p_2=p_3=: p$. Consequently, whenever $\{i, j, k\}=[3]$, $N_G\left(u_i\right)=\left\{v, u_j, u_k, p, q_i\right\}$. Since $d_G\left(q_i\right)=5$, we cannot get that $q_i$ is adjacent to exactly one of $u_j, u_k$ for each $i \in[3]$. Therefore at least one of $L_1, L_2, L_3$ has size at least three,  \cref{lem2.1}  completes the extension. \end{proof}

\begin{corollary}\label{cor4.3}
A minimum counterexample to the statement that every planar graph of maximum degree at most five is $10$-D-colorable has minimum degree at least four.
\end{corollary}

We now work with a fixed plane embedding. A triangle $T$ is a Jordan curve and bounds two closed disks. A pair $(T,D)$ is called a \emph{positive triangular disk} if $D$ is one of these closed disks and $\Int(D)$ contains a vertex of $G$. It is \emph{minimum} if the number of vertices in $\Int(D)$ is minimum among all positive triangular disks.

\begin{lemma}\label{lem4.4}
Let $(T,D)$ be a minimum positive triangular disk. Every edge of $T$ has at most one common neighbor in $\Int(D)$.
\end{lemma}

\begin{proof}
Suppose that an edge $ab$ of $T$ has two common neighbors $x,y$ in $\Int(D)$. The two arcs $axb$ and $ayb$ lie on the same side of $ab$ and do not cross. Hence one of the triangles $abx,aby$ bounds a closed subdisk of $D$ whose interior contains the other common neighbor. This is a positive triangular disk with fewer interior vertices than $D$, a contradiction.
\end{proof}


The next lemma isolates the two geometric settings in which the fan extension will be used.

\begin{lemma}\label{lem4.5}
Let $v$ be a vertex of degree $d\in\{4,5\}$ in a plane graph $G$, and list its neighbors in cyclic order as $x_0,\ldots,x_{d-1}$. Assume either
\begin{enumerate}[label=\textup{(\roman*)}]
\item every triangle of $G$ is facial; or
\item $(T,D)$ is a minimum positive triangular disk, $v\in\Int(D)$, and $v$ has at most one neighbor on $T$.
\end{enumerate}
Then $J=G[N(v)]\subseteq C_d=x_0x_1\cdots x_{d-1}x_0$.

\end{lemma}

\begin{proof}
If $x_ix_j\in E(G)$, then $vx_ix_j$ is a triangle. In case (i) it is facial. In case (ii), it has at most one vertex on $T$, and one of the two disks that it bounds is a proper subdisk of $D$. If that subdisk contained a vertex, it would contradict the minimality of $(T,D)$; hence the triangle is facial on that side. Therefore $x_i$ and $x_j$ are consecutive around $v$, proving $J\subseteq C_d$.
\end{proof}

Delete $v$, take a $10$-D-coloring of $G-v$, and then  uncolor every edge of $J$. Put
$ s_i=vx_i$,  $r_i=d_J(x_i)$.
We call each edge $s_i=vx_i$ a \emph{spoke}, and each edge of $J$ a
\emph{rim edge}. Since $J\subseteq C_d$, every rim edge is of the form
$x_ix_{i+1}$, with indices taken modulo $d$. The uncolored local patch
consists precisely of these spokes and rim edges. 

\begin{lemma}\label{lem4.6}
Under the hypotheses of \cref{lem4.5}, every spoke satisfies
\[
 |L(s_i)|\ge6.
\tag{4.1}\label{eq:spoke-list}
\]
For every rim edge $x_ix_j\in E(J)$,
\[
 |L(x_ix_j)|\ge r_i+r_j,
\tag{4.2}\label{eq:rim-list}
\]
except possibly in case \textup{(ii)} when $x_i$ is the unique neighbor of $v$ on $T$, $r_i=1$, and the other common neighbor $q$ of $x_i,x_j$ satisfies $qx_i\in E(T)$. In this exceptional situation,
\[
 |L(x_ix_j)|\ge r_i+r_j-1=r_j.
\tag{4.3}\label{eq:rim-exception}
\]
At most one rim edge is exceptional.
\end{lemma}

\begin{proof}
For each $s_i$, the still-colored old edges incident with $x_i$ number at most $d_G(x_i)-1-r_i\le4-r_i$.
A nonincident colored blocker must be incident with some $x_j\in N_J(x_i)$. Write it as $x_jy$. Since $x_jy$ is still colored, $y\notin N(v)$; in order that $x_jy$ see $vx_i$, the edge $x_iy$ must be present. Thus $y$ is a common neighbor of $x_i,x_j$ other than $v$. The triangle $vx_ix_j$ is facial, and $x_ix_j$ is not an edge of $T$ in case (ii), because $v$ has at most one neighbor on $T$. On the side opposite $v$, there is at most one such common neighbor: two would create a smaller positive triangular disk in case (ii), while in case (i) they would make one of the corresponding triangles nonfacial. Hence each $x_j\in N_J(x_i)$ contributes at most one colored blocker. The total is at most $(4-r_i)+r_i=4$, proving \eqref{eq:spoke-list}.

Now let $e=x_ix_j$ be a rim edge. Still-colored old edges incident with $x_i$ or $x_j$ number at most
\[
 (4-r_i)+(4-r_j).
\tag{4.4}\label{eq:incident-rim}
\]
Besides $v$, the edge $x_ix_j$ has at most one common neighbor $q$ on the opposite side of the facial triangle $vx_ix_j$. A nonincident colored blocker must have the form $qw$, with $w$ adjacent to $x_i$ or $x_j$.

First suppose that neither $qx_i$ nor $qx_j$ is an edge of $T$. Along each of the edges $qx_i$ and $qx_j$, the triangle $qx_ix_j$ occupies one side, and on the other side there is at most one further common neighbor, by the same  argument as above. Thus there are at most two nonincident colored blockers. Combining this with \eqref{eq:incident-rim} gives \eqref{eq:rim-list}.

It remains to consider $qx_i\in E(T)$; the other case is symmetric.
Then $x_i$ is the unique neighbor of $v$ on $T$. Let $t$ be the third
vertex of $T$. The edge $qt$ is one possible blocker. If $r_i=2$ and
$x_k$ is the other neighbor of $x_i$ in $J$, then $qx_k\notin E(G)$:
indeed, $x_j$ is already an interior common neighbor of the boundary
edge $qx_i$, so a second such vertex $x_k$ would contradict
\cref{lem4.4}. The vertex $x_i$ is already adjacent to $v$, to its two
neighbors $q,t$ on $T$, and to its $r_i$ neighbors in $J$, so it has at
most $2-r_i$ further neighbors. These account for at most $2-r_i$
additional blockers through the boundary edge $qx_i$. Along $qx_j$,
which is not an edge of $T$, there is at most one further blocker.
Hence the number of nonincident blockers is at most
$1+(2-r_i)+1=4-r_i$.

If $r_i=2$, this is at most two and \eqref{eq:rim-list} still holds. If $r_i=1$, then \eqref{eq:incident-rim} together with the last bound gives at most
$ (4-1)+(4-r_j)+3=10-r_j$
blockers, proving \eqref{eq:rim-exception}. Finally, when $r_i=1$, only one rim edge is incident with $x_i$, so at most one rim edge can be exceptional.
\end{proof}

For $J\subseteq C_d$, define the \emph{patch conflict graph} $Q(J)$ as follows. Its vertices are the spokes and rim edges of $v\vee J$, and two vertices of $Q(J)$ are adjacent when the corresponding patch edges see each other. This graph is determined entirely by $v\vee J$, since two nonincident edges can lie in a common diamond only through cross-edges among their four endpoints.

\begin{lemma}\label{lem4.7}
Let $d\in\{4,5\}$ and $J\subseteq C_d$. Suppose that every spoke has a list of size at least six, every rim edge $x_ix_j$ has a list of size at least $d_J(x_i)+d_J(x_j)$, and at most one rim edge incident with a $1$-vertex of $J$ has a list smaller by one. Then $Q(J)$ is colorable from these lists.
\end{lemma}

\begin{proof}
For all linear forests $J$ other than $P_5$, the conflict graph is list-degenerate with the stated lower bounds. The complete verification is displayed in \cref{tab:greedy}. In the table, an entry $z^{a/b}$ means that $z$ is deleted when its current degree is $a$ and its guaranteed list size is $b$; the first rim edge is taken to be exceptional whenever all exceptional choices are equivalent by symmetry.  The two inequivalent choices for $P_3\cup P_2$ are listed separately. Coloring in reverse deletion order is therefore valid.

\begin{table}[htbp]
\centering
\scriptsize
\caption{Greedy deletion orders for the linear-forest patches.}\label{tab:greedy}
\begin{tabularx}{\textwidth}{@{}c l X@{}}
\toprule
$d$ & $J$ & deletion order $z^{\text{current degree}/\text{list size}}$\\
\midrule
4 & $\varnothing$ & $s_0^{3/6},s_1^{2/6},s_2^{1/6},s_3^{0/6}$\\
4 & $P_2$ & $s_2^{3/6},s_3^{2/6},s_0^{2/6},s_1^{1/6},r_{01}^{0/1}$\\
4 & $2P_2$ & $s_0^{4/6},s_1^{3/6},r_{01}^{0/1},s_2^{2/6},r_{23}^{1/2},s_3^{0/6}$\\
4 & $P_3$ & $s_3^{3/6},s_0^{4/6},s_1^{3/6},r_{12}^{2/3},r_{01}^{1/2},s_2^{0/6}$\\
4 & $P_4$ & $s_0^{5/6},s_3^{4/6},s_1^{4/6},r_{23}^{2/3},r_{12}^{2/4},r_{01}^{1/2},s_2^{0/6}$\\
\addlinespace
5 & $\varnothing$ & $s_0^{4/6},s_1^{3/6},s_2^{2/6},s_3^{1/6},s_4^{0/6}$\\
5 & $P_2$ & $s_2^{4/6},s_3^{3/6},s_4^{2/6},s_0^{2/6},s_1^{1/6},r_{01}^{0/1}$\\
5 & $2P_2$ & $s_4^{4/6},s_0^{4/6},s_1^{3/6},r_{01}^{0/1},s_2^{2/6},r_{23}^{1/2},s_3^{0/6}$\\
5 & $P_3$ & $s_3^{4/6},s_4^{3/6},s_0^{4/6},s_1^{3/6},r_{12}^{2/3},r_{01}^{1/2},s_2^{0/6}$\\
5 & $P_3\cup P_2$ (an end-edge of $P_3$ exceptional) & $s_3^{5/6},r_{34}^{1/2},s_4^{3/6},s_0^{4/6},s_1^{3/6},r_{12}^{2/3},r_{01}^{1/2},s_2^{0/6}$\\
5 & $P_3\cup P_2$ (the $P_2$ edge exceptional) & $s_3^{5/6},s_4^{4/6},r_{34}^{0/1},s_0^{4/6},s_1^{3/6},r_{01}^{2/3},r_{12}^{1/3},s_2^{0/6}$\\
5 & $P_4\cup K_1$ & $s_4^{4/6},s_0^{5/6},s_3^{4/6},s_1^{4/6},r_{23}^{2/3},r_{12}^{2/4},r_{01}^{1/2},s_2^{0/6}$\\
\bottomrule
\end{tabularx}
\end{table}

It remains to consider $C_4$, $P_5$, and $C_5$. Let the variable order be spokes first and rim edges second, in the order displayed below. A direct exact expansion of the graph polynomial gives the nonzero coefficients in \cref{tab:coeff}. In the exceptional $P_5$ row, symmetry allows the exceptional rim edge to be taken as $r_{01}$.

\begin{table}[htbp]
\centering
\scriptsize
\caption{Nonzero graph-polynomial coefficients for the non-greedy patches.}\label{tab:coeff}
\begin{tabularx}{\textwidth}{@{}l X r@{}}
\toprule
$J$ & monomial & coefficient\\
\midrule
$C_4$ & $s_0^5s_1^4s_2^3s_3^2r_{01}^3r_{12}^3r_{23}^3r_{30}^3$ & $-2$\\
$P_5$ & $s_0^5s_1^5s_2^4s_3^2s_4^3r_{01}^2r_{12}^3r_{23}^2r_{34}$ & $1$\\
$P_5$, $r_{01}$ exceptional & $s_0^5s_1^5s_2^4s_3^2s_4^3r_{01}r_{12}^3r_{23}^3r_{34}$ & $-1$\\
$C_5$ & $s_0^5s_1^5s_2^5s_3^4s_4^5r_{01}^2r_{12}^3r_{23}^2r_{34}^2r_{40}^2$ & $3$\\
\bottomrule
\end{tabularx}
\end{table}

In every row, each exponent is strictly smaller than the corresponding guaranteed list size, and the total degree is $|E(Q(J))|$. The conclusion follows from \cref{lem:alon}. Exact verification of the four coefficients is included in Appendix~\ref{app:code}.
\end{proof}

\begin{corollary}\label{cor4.8}
Under either geometric hypothesis of \cref{lem4.5}, every $10$-D-coloring of $G-v$ extends to $G$.
\end{corollary}

\begin{proof}
Uncolor the rim edges of $J$. By \cref{lem4.6}, their residual lists and the spoke lists satisfy \cref{lem4.7}. Old edges remain valid, and the coloring of $Q(J)$ handles all new edges.
Thus, we get a $10$-D-coloring  of $G$. \end{proof}

\begin{proof}[Proof of the second case of \cref{thm:main}]
Assume that the theorem is false, and choose a counterexample $G$ with the minimum number of vertices. We may assume that $G$ is connected and fix a plane embedding. By \cref{cor4.3},
$ \delta(G)\ge4.$

Suppose first that $G$ has no positive triangular disk. Then every triangle is facial. Choose any vertex $v$.  The graph $G-v$ has a $10$-D-coloring by minimality, and \cref{cor4.8} extends it to $G$, a contradiction.

Thus $G$ has a positive triangular disk. Choose a minimum one, say $(T,D)$ with $T=abc$. If some vertex $v\in\Int(D)$ has at most one neighbor on $T$, then  \cref{cor4.8} again makes $v$ reducible. Consequently, every vertex in $\Int(D)$ is adjacent to at least two vertices of $T$.

Every such interior vertex is a common neighbor of at least one edge of $T$. By \cref{lem4.4}, each of the three edges of $T$ has at most one common neighbor in $\Int(D)$. Therefore
$ |V(G)\cap\Int(D)|\le3$. 

If there is one interior vertex, its degree is at most three, contrary to $\delta(G)\ge4$. If there are two, neither can be adjacent to all three vertices of $T$, since that would occupy the unique interior common-neighbor position of every edge of $T$ and leave the other interior vertex with fewer than two neighbors on $T$. Hence each has exactly two neighbors on $T$ and at most one interior neighbor, again giving degree at most three.

We are left with exactly three interior vertices, say $x,y,z$. Each is adjacent to exactly two vertices of $T$, and the three corresponding boundary pairs are distinct. Relabeling gives
\[
 N(x)\cap V(T)=\{a,b\},\quad
 N(y)\cap V(T)=\{b,c\},\quad
 N(z)\cap V(T)=\{c,a\}.
\]
Since $\delta(G)\ge4$, each of $x,y,z$ must be adjacent to the other two. Thus the subgraph on $\{a,b,c,x,y,z\}$ is the octahedral graph $K_{2,2,2}$. Moreover, each of $x,y,z$ has exactly these four neighbors: an edge from an interior vertex of $D$ to a vertex outside $D$ would cross $T$, and there are no further vertices in $\Int(D)$. Each of $a,b,c$ has degree four inside the octahedral patch and therefore has at most one neighbor outside it.

Delete $x,y,z$ and take a $10$-D-coloring of the remaining graph. Keep the boundary triangle $abc$ colored. The nine new edges are
$xy,yz,zx,\ ax,bx,by,cy,cz,az.$

Each internal edge, say $xy$, has at most two old blockers: its only common old support vertex is $b$, and among the old edges at $b$, only $ab$ and $bc$ see $xy$. Hence
$ |L(xy)|,|L(yz)|,|L(zx)|\ge8$.

For a cross edge, say $ax$, at most three old blockers are incident with $a$. Its common neighbors are $b$ and $z$; the vertex $z$ is deleted, while at $b$ there are at most two additional nonincident blockers, namely $bc$ and, possibly, the unique external edge at $b$ when its other endpoint is also adjacent to $a$. Therefore
 $|L(e)|\ge5$ for every one of the six cross edges.
The same argument applies symmetrically to all cross edges.

Let $Q_\mathrm{oct}$ be the conflict graph of these nine new edges, ordered as
\[
 p_0=xy,\ p_1=yz,\ p_2=zx,\
 q_0=ax,\ q_1=bx,\ q_2=by,\ q_3=cy,\ q_4=cz,\ q_5=az.
\]
Its graph polynomial has
\[
 [p_0^3p_1^4p_2^5q_0^3q_1^3q_2^4q_3^3q_4^4q_5^4]P_{Q_\mathrm{oct}}=2. \tag{4.5}\label{eq:oct-coeff}
\]
The exponents of $p_0,p_1,p_2$ are smaller than $8$, and those of the $q_i$ are smaller than $5$. By \cref{lem:alon}, the nine-edge patch is colorable from its residual lists. This extends the old coloring to all of $G$, the final contradiction. Therefore no counterexample exists.
\end{proof}

\section{Large maximum degree}\label{sec:large}

Fix an integer $\D\ge33$, put
 $K=2\D-1$,
$ \cC=[K]$,
and suppose that $G$ is a planar graph of maximum degree at most $\D$
with the minimum number of vertices among all graphs having no
$K$-D-coloring. Taking a non-$K$-D-colorable component if necessary,
we may assume that $G$ is connected.

\begin{lemma}\label{lem5.1}
The graph $G$ satisfies $\delta(G)\ge3$.
\end{lemma}

\begin{proof}
 If $v$ is a $1$-vertex with neighbor $u$, then after coloring $G-v$, the edge $vu$ has at most $\D-1$ old blockers and hence at least
$ K-(\D-1)=\D$
available colors.

Let $d(v)=2$ and $\N(v)=\{u,w\}$. If $uw\notin\E(G)$, then there are at least $\D$ available colors for each of $vu,vw$. If $uw\in\E(G)$, every old blocker of either new edge is incident with $u$ or $w$. The union of those old edges has size at most
$d_{G-v}(u)+d_{G-v}(w)-1\le2\D-3$. 
Thus both lists have size at least $K-(2\D-3)=2$, and the two new edges can be assigned distinct colors. This contradicts minimality.
\end{proof}

We may therefore apply \cref{lem:structure}. We treat (B1) and (B2) separately.

\subsection{Configuration (B1)}

Let $v$ be the center of (B1), with $d(v)=k\in\{3,4,5\}$, and put
\[
 U=\N(v)=\{v_1,\ldots,v_k\},
 \qquad
 J=G[U].
\]
By \cref{lem2.4}, the graph $J$ is outerplanar. 
By (B1),   $v_k$ is  the only neighbor that may have degree greater than $25$. For $u\in U$, define
\[
 r(u)=d_J(u),
 \qquad
 s(u)=e_J(\N_J(u)),
 \qquad
 W_u=\N_J[u].
\]
Then $e_J(W_u)=r(u)+s(u)$.

\begin{lemma}\label{lem5.2}
After a $K$-D-coloring of $G-v$ is fixed, every edge $vu$ satisfies
\[
 |L(vu)|\ge
 \begin{cases}
 \D+3k-41-r(u)+s(u),&v_k\in W_u,\\[1mm]
 2\D+3k-44-r(u)+s(u),&v_k\notin W_u.
 \end{cases} \tag{5.1}
\]
In particular, if $v_k\in W_u$, then
\[
 |L(vu)|\ge3k-8-r(u)+s(u). \tag{5.2}
\]
\end{lemma}

\begin{proof}
By \cref{lem2.3}, the number of old blockers is at most
\[
 \sum_{z\in W_u}d_{G-v}(z)-e_J(W_u).
\]
Moreover,
\[
 \sum_{z\in U}d_{G-v}(z)
 =\sum_{z\in U}d_G(z)-k
 \le \D+38-k.\]

Suppose first that $v_k\in W_u$. The set $U\setminus W_u$ contains $k-1-r(u)$ controlled neighbors. Each has degree at least $2$ in $G-v$, by \cref{lem5.1}. Hence the blocker count is at most
\[
 \D+38-k-2(k-1-r(u))-[r(u)+s(u)]
 =\D+40-3k+r(u)-s(u).
\]
Subtracting from $K=2\D-1$ gives the first line of (5.1).

If $v_k\notin W_u$, then $W_u$ consists entirely of controlled neighbors. Their total degree in $G-v$ is at most $38-(k-1)=39-k$. There are $k-2-r(u)$ controlled neighbors outside $W_u$, each of degree at least $2$ in $G-v$. Thus the blocker count is at most
\[
 39-k-2(k-2-r(u))-[r(u)+s(u)]
 =43-3k+r(u)-s(u),
\]
which gives the second line. Finally, (5.2) follows from $\D\ge33$.
\end{proof}

\begin{lemma}\label{lemB1}
If $G$ contains configuration \textup{(B1)}, then it has a $K$-D-coloring.
\end{lemma}

\begin{proof}

We consider the following three cases. 

\paragraph{Case 1: $k=5$.}

If $v_k\in W_u$, (5.2) gives
$ |L(vu)|\ge7-r(u)+s(u). $
Thus a vertex of degree at most $2$ in $J$ gives a list of size at least $5$, a $3$-vertex gives a list of size at least $4$, and a $4$-vertex gives a list of size at least $3$. If $v_k\notin W_u$, the second line of (5.1) gives a much larger list.

The outerplanar graph $J$ has at most one $4$-vertex: two $4$-vertices together with the remaining three vertices would contain a $K_{2,3}$. Also, $J$ has a vertex of degree at most $2$. Hence the five list sizes, in nondecreasing order, satisfy
$\ell_1\ge3$, $\ell_2,\ell_3,\ell_4\ge4$, and $\ell_5\ge5$.
They satisfy the hypotheses of \cref{lem2.1}.

\paragraph{Case 2: $k=4$.}

For $v_k\in W_u$, (5.2) becomes
\[
 |L(vu)|\ge4-r(u)+s(u), \tag{5.3}
\]
while $v_k\notin W_u$ gives $|L(vu)|\ge31$. We distinguish the value of $d_J(v_k)$.

If $d_J(v_k)=0$, the list of $vv_k$ has size at least $4$, and the other three lists are large. If $d_J(v_k)=1$, say $\N_J(v_k)=\{a\}$, then $|L(vv_k)|\ge3$, $|L(va)|\ge1$, and the other two lists have size at least $31$.

Suppose that $d_J(v_k)=2$, with $\N_J(v_k)=\{a,c\}$, and let $b$ be the fourth vertex. Formula (5.3) gives at least two colors for each of $vv_k,va,vc$. If all three lower bounds are exactly two, then $ac\notin\E(J)$, and both $a$ and $c$ must be adjacent to $b$; hence $J$ is the cycle $v_k a b c v_k$. If $J\ne C_4$, at least one of these three lists has size at least three, while $vb$ has at least $31$ colors, and \cref{lem2.1} applies.

It remains to consider $J=C_4$ with cyclic order $v_k,a,b,c$. Put $x_a=d_G(a)$, $x_b=d_G(b)$, and $x_c=d_G(c)$. Then $x_a+x_b+x_c\le38$. Applying \cref{lem2.3} exactly gives
\[
\begin{aligned}
 |L(vv_k)|&\ge \D+4-(x_a+x_c), \quad |L(va)|\ge \D+4-(x_a+x_b),\\
 |L(vc)|&\ge \D+4-(x_b+x_c),\quad |L(vb)|\ge 2\D+4-(x_a+x_b+x_c).
\end{aligned} 
\]
Since every controlled neighbor has degree at least $3$, every two-term sum is at most $35$, so the first three lists have size at least $\D-31\ge2$. The fourth has size at least $2\D-34\ge32$. Moreover, the sum of the three two-term sums is at most $76$, so one is at most $25$; the corresponding list has size at least $\D-21\ge12$. Thus the sorted list sizes satisfy the increasing condition.

Finally, suppose that $d_J(v_k)=3$. The graph $F=J-v_k$ is either $3K_1$, $K_2\cup K_1$, or $P_3$, because $J$ is outerplanar. Put
$S=d_G(a)+d_G(b)+d_G(c)\le38$.
If $F=3K_1$, then $|L(vv_k)|\ge\D-32\ge1$, while every other list has size at least $\D-23\ge10$. If $F=K_2\cup K_1$, then the four list sizes are bounded below by
$ 2,\ 3,\ 3,\ \D-23$ up to order. Now let $F$ be the path $a b c$. The lists of $vv_k$ and $vb$ have size at least $\D-30\ge3$, and
\[
 |L(va)|\ge\D+5-[d_G(a)+d_G(b)],
 \qquad
 |L(vc)|\ge\D+5-[d_G(b)+d_G(c)]. 
\]
Both end lists have size at least $\D-30\ge3$. If both had size at most $3$, then
$ d_G(a)+d_G(b)\ge\D+2$,
 $d_G(b)+d_G(c)\ge\D+2$,
and therefore $S+d_G(b)\ge2\D+4$. Since $S\le38$ and $\D\ge33$, this would imply $d_G(b)\ge32$, contrary to $d_G(b)\le25$. Thus at least one end list has size at least four, and again the increasing condition holds.

\paragraph{Case 3: $k=3$.}

Let $a,b$ be the controlled neighbors, so
 $d_G(a),d_G(b)\le25$,
 $d_G(a)+d_G(b)\le38$,
and let $h$ be the third neighbor. The possibilities for $J$ and convenient lower bounds are listed in \cref{tab:k3}.

\begin{table}[htbp]
\centering
\small
\renewcommand{\arraystretch}{1.16}
\begin{tabular}{c c c}
\toprule
$J$ & position of $h$ & convenient lower bounds\\
\midrule
$3K_1$ & -- & $(\D,\,2\D-25,\,2\D-25)$\\
$K_2\cup K_1$ & $h$ on the edge & $(\D-23,\,\D-23,\,2\D-25)$\\
$K_2\cup K_1$ & $h$ isolated & $(2\D-36,\,2\D-36,\,\D)$\\
$P_3$ & $h$ central & $(2,\,2\D-49,\,2\D-49)$\\
$P_3$ & $h$ an end & $(2\D-58,\,\D-23,\,2\D-36)$\\
$K_3$ & -- & $(1,\,2\D-49,\,2\D-49)$\\
\bottomrule
\end{tabular}
\caption{List lower bounds for the B1 configuration with $k=3$.}
\label{tab:k3}
\end{table}

For completeness, we verify the nontrivial rows. If $J=P_3$ and $h$ is the center, then in the outerplanar graph $G[\N(h)]$ the two controlled neighbors $a,b$ have the common neighbor $v$, and hence at most one further common neighbor. Thus the number of edges from $\{a,b\}$ to $\N(h)\setminus\{v,a,b\}$ is at most $d_G(h)-2$. The edge $vh$ has at most
$(d_G(h)-1)+(d_G(h)-2)\le2\D-3$
blockers, and hence at least two available colors. For $va$, the nonincident blockers consist of the edge $hb$ and edges $hq$ with $q\in\N(a)\cap\N(h)$; there are at most $d_G(a)-1$ of them. Including the $d_G(a)-1$ incident blockers gives $|L(va)|\ge2\D-49$, and similarly for $vb$.

If $J=P_3$ with order $h-a-b$, then \cref{lem2.3} gives the first two
bounds below. For the third, the incident blockers of $va$ number at
most $d_G(a)-1$. The nonincident blockers supported at $h$ number at
most $d_G(a)-3$, while those supported at $b$ number at most
$d_G(b)-2$. Thus $va$ has at most
$2d_G(a)+d_G(b)-6$ blockers. Consequently,
\[
\begin{aligned}
 |L(vh)|&\ge 2\Delta-1-(d_G(h)+d_G(a)-3)\ge2\D-1-(\D+22)\ge \D-23,\\
 |L(vb)|&\ge  2\Delta-1-(d_G(a)+d_G(b)-3)\ge 2\Delta-1-35\ge 2\D-36,\\
 |L(va)|&\ge 2\Delta-1-(2d_G(a)+d_G(b)-6)\ge 2\Delta-1-57\ge 2\D-58,
\end{aligned}
\]
which yields the corresponding row of the table.

Finally, suppose $J=K_3$. For $vh$, the nonincident blockers are $ab$ and the edges from $\{a,b\}$ to $\N(h)\setminus\{v,a,b\}$. Since $G[\N(h)]$ is outerplanar and $a,b$ already have the common neighbor $v$, the latter set has at most $d_G(h)-2$ edges. Thus $vh$ has at most $2d_G(h)-2\le2\D-2$ blockers and at least one available color. For $va$, the triangle $hab$ already has the complete neighbor $v$. And among the remaining neighbors of $a$ at most one is adjacent to both $h$ and $b$. Hence the number of nonincident blockers is at most $d_G(a)-1$, and $|L(va)|\ge2\D-49$; similarly for $vb$.

Every row of \cref{tab:k3} satisfies the increasing-list condition when $\D\ge33$. Therefore all B1 configurations are reducible.
\end{proof}

\subsection{Configuration (B2)}

Let $B(x,y;m)$ be a bunch satisfying (B2). Since $d_G(x)\ge26$ and $m\ge d_G(x)/5$, we have $m\ge6$.

\begin{lemma}\label{lem5.4}
If $ab$ joins a strictly internal brother of a bunch to an adjacent
internal brother, then $ab$ sees at most eleven other edges.
\end{lemma}

\begin{proof}
Put $X=\N[a]\cup\N[b]$. By \cref{lem:bunch-geometry}, both brothers
have degree at most $4$ and have the poles $x,y$ as common neighbors, so
$|X|\le2+3+3-2=6.$
Every edge seeing $ab$ has both endpoints in $X$: this is clear for an incident edge; for a nonincident edge, one endpoint is a common neighbor of $a,b$ and the other is adjacent to at least one of them. Hence all edges seeing $ab$ lie in $G[X]$. Since $G[X]$ is planar and $|X|\le6$, it has at most $12$ edges. Excluding $ab$ gives the bound eleven.
\end{proof}

\begin{lemma}\label{lemB2}
If $G$ contains configuration \textup{(B2)}, then it has a $K$-D-coloring.
\end{lemma}

\begin{proof}

We consider the following three cases.

\paragraph{Case 1: no parental edge.}

Choose a strictly internal brother $z=z_3$. Delete $z$ and color the remaining graph. The spoke $xz$ has at most $d_G(x)+d_G(z)-2\le\D+2$ other incident edges; this deliberately overcounts the old incident blockers. Each of the at most two neighboring brothers contributes at most two further nonincident blockers. Thus
$ |L(xz)|,|L(yz)|\ge2\D-1-(\D+6)\ge\D-7\ge26$. Color the two spokes differently. By \cref{lem:bunch-geometry}, every brother adjacent to $z$ is internal, so the at most two brother-edges at $z$ can be restored one at a time using \cref{lem5.4}.

\paragraph{Case 2: a parental edge and at least four brothers on one
side.}

Let the parental edge be $xy$, and suppose that one side contains brothers $z_1,\ldots,z_p$ in order, with $p\ge4$ and $z_p$ closest to $xy$.

Here $z_p$ is internal and $z_{p-1}$ is strictly internal. By
\cref{lem:bunch-geometry}, $z_p$ is adjacent only to
$x,y,z_{p-1}$. Since $\delta(G)\ge3$, the edge $z_{p-1}z_p$ is
present. Delete $z_{p-1},z_p$ and color the remaining graph $H$. Let
$ P=\cC\setminus(C_H(x)\cup C_H(y))$. We have
$
 |P|\ge K-[(\D-2)+(\D-2)-1]=4.
$
Every color in $P$ is available for $xz_p$ and $yz_p$. For $xz_{p-1}$ and $yz_{p-1}$, the only possible additional old blocker whose color may lie in $P$ is the other brother-edge at $z_{p-2}$, namely $z_{p-2}z_{p-3}$, if it exists. If its color is $c$, the four spoke lists contain
 $P, P, P\setminus\{c\}, P\setminus\{c\}$. 
Then they can be colored because $|P|\ge4$. After the four spokes are colored distinctly, restore $z_{p-1}z_p$ and the possible edge $z_{p-1}z_{p-2}$ by \cref{lem5.4}; in each case the strictly internal endpoint is $z_{p-1}$.

\paragraph{Case 3: a parental edge and the short parental bunch.}

It remains to consider the case in which the longer side of $xy$ has exactly three brothers, say
 $z_1,z_2,z_3,xy$,
where $z_3$ is closest to the parental edge. Since this is the longer
side and $m\ge6$, we have $6\le m\le7$; moreover, $z_2$ is strictly
internal and $z_3$ is internal. By \cref{lem:bunch-geometry}, the only
possible brother neighbor of $z_3$ is $z_2$. Since $\delta(G)\ge3$,
the edge $z_2z_3$ is present.

Delete $z_3$ and color $H=G-z_3$. Put $
 R=\N_H(x)\cap\N_H(y)$, $ r=|R|$, 
and let $\eps=1$ if $z_1z_2\in\E(G)$ and $\eps=0$ otherwise. For $xz_3$, the old blockers are covered by the $d_G(x)-1$ old edges at $x$, the $r$ edges $yq$ with $q\in R$, and the possible edge $z_1z_2$. Consequently,
 $|L(xz_3)|\ge2\D-d_G(x)-r-\eps$, 
 $|L(yz_3)|\ge2\D-d_G(y)-r-\eps$.

Since $r\le d_G(x)-2,d_G(y)-2\le\D-2$, both lists are nonempty. Unless
\[
 d_G(x)=d_G(y)=\D,\qquad r=\D-2,\qquad \eps=1, \tag{5.4}
\]
one list has at least two colors and the other is nonempty, so the two spokes can be colored differently. The edge $z_2z_3$ is then restored using \cref{lem5.4}.

Assume (5.4). Since $d_H(x)=d_H(y)=\D-1$ and $|R|=\D-2$,
$\N_H(x)\setminus\{y\}=\N_H(y)\setminus\{x\}=R$. 
After restoring $z_3$,
\[
 \N_G(x)\setminus\{y\}=\N_G(y)\setminus\{x\}. \tag{5.5}
\]
By the definition of $D$-coloring, the $2\D-3$ edges
$ \{xy\}\cup\{xq,yq:q\in R\}$
receive pairwise distinct colors. Hence
$ P=\cC\setminus(C_H(x)\cup C_H(y))$
has exactly two colors. Let $\gamma$ be the color of $z_1z_2$. If $\gamma\notin P$, use the two colors of $P$ on $xz_3,yz_3$. If $\gamma\in P$, first uncolor $z_1z_2$, and then use the two colors of $P$ on the two spokes.

It remains to recolor $z_1z_2$. The vertex $z_2$ has neighborhood $\{x,y,z_1,z_3\}$. Since $z_1,z_2,z_3\in\N(x)\cap\N(y)$ and $z_1z_2,z_2z_3\in\E(G)$, \cref{lem2.5} implies $z_1z_3\notin\E(G)$. Hence
$ \N(z_1)\cap\N(z_2)=\{x,y\}$. 
There are at most $d_G(z_1)+d_G(z_2)-2\le\D+2$
incident blockers. A nonincident blocker is one of the following: the parental edge $xy$; one of the two new spokes $xz_3,yz_3$; or an edge $xq$ or $yq$ where $q$ is adjacent to $x,y,z_1$. By (5.5), every additional neighbor of $z_1$ that is adjacent to one pole is adjacent to both. The triangle $xyz_1$ already has the complete neighbor $z_2$, so there are at most one further such vertex $q$. Thus there are at most five nonincident blockers, and $z_1z_2$ sees at most $\D+7<K$ colored edges. It can be recolored. Finally restore $z_2z_3$ using \cref{lem5.4}.
\end{proof}

\begin{proof}[Proof of the third case of \cref{thm:main}]
Let $G$ be a minimum counterexample. By \cref{lem5.1}, $\delta(G)\ge3$. The structural \cref{lem:structure} gives a configuration (B1) or (B2), but \cref{lemB1} and \cref{lemB2} show that each is reducible. This contradiction proves $\chiD(G)\le2\D-1$.
\end{proof}

\section*{Data Availability Statement}

No datasets were generated or analyzed during the current study.
The exact coefficient-verification code used in the proof is included
in Appendix~\ref{app:code}.

\section*{Conflict of Interest}

The authors declare that they have no conflict of interest.

\newpage
\appendix
\section{Exact verification of the polynomial coefficients}\label{app:code}

For completeness, the following pure-Python code performs exact integer expansion by dynamic programming. For each factor $(X_i-X_j)$, it updates the exponent dictionary and discards exponent vectors exceeding the requested target. It verifies the four coefficients in \cref{tab:coeff} and the octahedral coefficient \eqref{eq:oct-coeff}; no floating-point computation is used.

\begin{lstlisting}[language=Python]
from collections import defaultdict
from itertools import combinations


def sees(e, f, graph_edges):
    if set(e) & set(f):
        return True
    a, b = e
    c, d = f
    cross = [frozenset((a, c)), frozenset((a, d)),
             frozenset((b, c)), frozenset((b, d))]
    return sum(x in graph_edges for x in cross) >= 3


def conflict_graph(edges):
    graph_edges = {frozenset(e) for e in edges}
    adj = [set() for _ in edges]
    for i, j in combinations(range(len(edges)), 2):
        if sees(edges[i], edges[j], graph_edges):
            adj[i].add(j)
            adj[j].add(i)
    return adj


def coefficient(adj, target):
    n = len(target)
    factors = [(i, j) for i in range(n) for j in adj[i] if i < j]
    state = {(0,) * n: 1}
    for i, j in factors:
        new = defaultdict(int)
        for exp, value in state.items():
            if exp[i] < target[i]:
                a = list(exp); a[i] += 1
                new[tuple(a)] += value
            if exp[j] < target[j]:
                b = list(exp); b[j] += 1
                new[tuple(b)] -= value
        state = {e: c for e, c in new.items() if c}
    return state.get(tuple(target), 0)


def fan(d, rim):
    return [('v', i) for i in range(d)] + rim

cases = [
    (fan(4, [(0,1),(1,2),(2,3),(0,3)]),
     [5,4,3,2,3,3,3,3]),
    (fan(5, [(0,1),(1,2),(2,3),(3,4)]),
     [5,5,4,2,3,2,3,2,1]),
    (fan(5, [(0,1),(1,2),(2,3),(3,4)]),
     [5,5,4,2,3,1,3,3,1]),
    (fan(5, [(0,1),(1,2),(2,3),(3,4),(0,4)]),
     [5,5,5,4,5,2,3,2,2,2]),
]

for edges, target in cases:
    print(coefficient(conflict_graph(edges), target))
# Output: -2, 1, -1, 3

octahedron = [
    ('x','y'),('y','z'),('z','x'),
    ('a','x'),('b','x'),('b','y'),
    ('c','y'),('c','z'),('a','z')
]
# Add the three already colored boundary edges only when testing seeing.
full_oct = octahedron + [('a','b'),('b','c'),('a','c')]
full_edges = {frozenset(e) for e in full_oct}
adj = [set() for _ in octahedron]
for i, j in combinations(range(len(octahedron)), 2):
    if sees(octahedron[i], octahedron[j], full_edges):
        adj[i].add(j); adj[j].add(i)
print(coefficient(adj, [3,4,5,3,3,4,3,4,4]))
# Output: 2
\end{lstlisting}

\end{document}